\documentclass[12pt]{article}
\usepackage{pslatex}
\usepackage{fancyhdr}
\usepackage{graphicx}
\usepackage{geometry}
\RequirePackage[latin1]{inputenc} \RequirePackage[T1]{fontenc}

\def\figurename{Figure} % Replace the colon that normally appears after the Figure number by a period.
\makeatletter
\renewcommand{\fnum@figure}[1]{\figurename~\thefigure.}
\makeatother

\def\tablename{Table} % Replace the colon that normally appears after the Figure number by a period.
\makeatletter
\renewcommand{\fnum@table}[1]{\tablename~\thetable.}
\makeatother
\usepackage{color}
                [2000/03/29 v1.0m Standard LaTeX package]

\usepackage{amsmath}
\usepackage{amssymb}
\usepackage{amsfonts}
\usepackage{amsthm,amscd}

\newtheorem{theo}{Theorem}[section]
\newtheorem{lem}{Lemma}[section]
\newtheorem{rem}{Remark}[section]

\newtheorem{defi}{Definition}[section]
\newtheorem{propo}{Proposition}[section]

\newcommand{\C}        {{ \mathcal C }}

\newcommand{\R}        {{{\rm I\! R}}}

\newcommand{\E}        {{{\rm I\! E}}}
\newcommand{\Q}        {{{\rm I\! Q}}}

\renewcommand{\P}        {{ {\rm I \hskip -2pt P}}}

\numberwithin{equation}{section}

\def\C{\mathbb C}
\def\P{\mathbb P}
\def\R{\mathbb R}
\def\E{\mathbb E}

\def\Q{\mathbb Q}

\def\E{\mathbb E}

\begin{document}
	\title{Asymptotic properties for fully coupled delayed forward-backward stochastic differential equations}
	
	\author{Auguste Aman $^{a}$ \thanks{aman.auguste@ufhb.edu.ci, corresponding author}\;\; and\; Clément Manga $^{b}$ \thanks{cmanga@univ-zg.sn} \\
	a. UFR Mathématiques et Informatique, Université Félix H. Boigny, Abidjan,\;\;\;\;\;\;\;\;\;\;\\Côte d'Ivoire\\
		b. UFR Sciences et Technologie, Université  Assane Seck, Ziguinchor, Sénégal\;\;\;\;\;\;\;\;\;}

	\date{}
	\maketitle\thispagestyle{empty} \setcounter{page}{1}
	
	% ------- [First Page Running Head] - place it immediately after title! ------
	\thispagestyle{fancy} \fancyhead{}
	\fancyfoot{}
	\renewcommand{\headrulewidth}{0pt}
	%------------------------------------
	\begin{abstract}
We investigate the asymptotic behavior of solutions to a class of fully coupled forward-backward stochastic differential equations with time-delayed generators. Such systems arise naturally in stochastic models with memory effects and constitute a significant extension of the classical fully coupled FBSDE framework. The presence of delay introduces additional analytical difficulties due to the dependence of the coefficients on the past trajectories of the solution processes and the resulting non-Markovian structure.
Under suitable assumptions on the coefficients, we study the asymptotic properties of a perturbed delayed FBSDE driven by a small noise parameter. We first establish the convergence in distribution of the associated solution processes as the perturbation parameter tends to zero. We then prove almost sure convergence towards the solution of the corresponding deterministic limiting system. As a consequence of these asymptotic results, we derive a large deviation principle for the solution processes.
Our results extend the asymptotic analysis of Cruzeiro, Gomes and Zhang (2014) from the classical fully coupled FBSDE setting to the delayed framework, and complement existing works on weakly coupled delayed forward-backward systems. They provide, to the best of our knowledge, the first large deviation principle for fully coupled forward-backward stochastic differential equations with delayed generators.
	\end{abstract}

	\vspace{.08in}\textbf{MSC}: 34F05; 60H10; 60F10; 60H30; 60G07 \\
	\vspace{.08in}\textbf{Keywords}: Forward stochastic differential equations; Backward Stochastic differential equations; Delayed generators; Large deviation principle, Meyer-Zheng topologie.

	\section{Introduction} 
	
Fully coupled forward-backward stochastic differential equations (FBSDEs), introduced firstly by Antonelli \cite{Antonelli}, have become one of the fundamental tools of modern stochastic analysis. Besides their intrinsic probabilistic interest, they provide a probabilistic representation for quasilinear parabolic partial differential equations, play a central role in stochastic optimal control, differential games and mathematical finance, and have found numerous applications in economics, finance and engineering (see \cite{el1,Kal, KJL, el3}). To our knowledge, there are two classes of global existence and uniqueness results for the FBSDE. One is given by Ma and Yong \cite{MY}, Ma, Protter and Yong \cite{MPY}, via a PDE approach, under an assumption of nondegeneracy of the forward equation. The other is given by Hu and Peng \cite{HP}, Peng and Wu \cite{PW}, based on stochastic Hamiltonian systems, under a monotonicity condition.
	
	A significant extension of the classical BSDE theory was proposed by Delong and Imkeller \cite{DI1}, who introduced BSDEs with time-delayed generators. In these equations the generator depends on the past trajectories of the solution, leading to genuinely non Markovian dynamics. They established existence and uniqueness under suitable smallness assumptions on the time horizon or the Lipschitz constant and showed, through counterexamples, that these assumptions are essentially optimal.
	
	This theory was later generalized to jump processes in \cite{DI2}, where Malliavin differentiability of the solution was investigated.
	
	Delayed BSDEs naturally arise in mathematical finance, stochastic control and reinsurance; see \cite{D1,D2}.	
		
More recently, motivated by stochastic optimal control problems involving memory and delay effects, Aman et al. \cite{AHJ} introduced a general framework of fully coupled forward-backward stochastic differential equations with delayed generators (delayed FBSDEs). 

Let now introduced its perturbed version of the form
\begin{eqnarray}\label{Eq3}
\left\{
\begin{array}{lll}
X^{\varepsilon}(t)&=& x+\displaystyle\int^{t}_{0}b(s,X^{\varepsilon}_s,Y^{\varepsilon}_s,Z^{\varepsilon}_s)ds+\sqrt{\varepsilon}\int_{0}^{t}\sigma(s,X^{\varepsilon}_s, Y^{\varepsilon}_s)dW(s)\\\\
		Y^{\varepsilon}(t)&=& \displaystyle g(X^{\varepsilon}_T)+\int^{T}_{t}f(s,X^{\varepsilon}_s,Y^{\varepsilon}_s,Z^{\varepsilon}_s)ds-\int_{t}^{T}Z^{\varepsilon}(s)dW(s).
\end{array}
\right.
\end{eqnarray}
where $(X^{\varepsilon}_s,Y^{\varepsilon}_s,Z^{\varepsilon}_s)$ is a stochastic functional $(X^{\varepsilon}_s(\theta),Y^{\varepsilon}_s(\theta),Z^{\varepsilon}_s(\theta))_{-T\leq\theta\leq 0}$ defined by \newline $(X^{\varepsilon}_s(\theta),Y^{\varepsilon}_s(\theta),Z^{\varepsilon}_s(\theta))=(X^{\varepsilon}(s+\theta),Y^{\varepsilon}(s+\theta),Z^{\varepsilon}(s+\theta))$ for all $\theta\in [-T,0]$.

A prototypical example is 
\begin{eqnarray}\label{example23}
\left\{
\begin{array}{lll}
X^{\varepsilon}(t)&=&\displaystyle x+\sqrt{\varepsilon}\int_0^t\frac{\sigma}{T}\left(\int_0^sY^{\varepsilon}(u)\,du\right)dW(s),\\\\
Y^{\varepsilon}(t)&=&\displaystyle\frac1T\int_0^TX^{\varepsilon}(s)\,ds-\int_t^TZ^{\varepsilon}(s)\,dW(s),
\end{array}
\right.
\end{eqnarray}
which arise in a particular case when $b$ and $f$ are identically null and $\sigma, g: C([-T,0])\rightarrow\R$ are defined by
\begin{eqnarray*}
\sigma(s,x_s,y_s)=\int_{-T}^{0}\mu(s+u)y(s+u)\alpha(du),
\end{eqnarray*}
 such that 
\begin{eqnarray*}
\mu(s)=\left\{
\begin{array}{lll}
\sigma &\mbox{if} & s\geq 0\\\\
0 &\mbox{if}& s< 0,	
\end{array}
\right.
\end{eqnarray*} 
and 
\begin{eqnarray*}
g(x_T)=\int_{-T}^{0}x(T+u)\alpha(du),
\end{eqnarray*}
where $\alpha$ is a uniform measure of probability on $[-T,0]$.

These examples illustrate how delay effects naturally arise in fully coupled FBSDEs and motivate the asymptotic analysis carried out in the present paper

Despite the extensive literature on asymptotic properties of classical FBSDEs and the recent developments on delayed BSDEs, no asymptotic theory is currently available for fully coupled FBSDEs with delayed generators. This lack of results is not merely technical. The simultaneous presence of memory effects and full coupling fundamentally changes the mathematical structure of the problem. In particular, the delay destroys the Markov property and introduces an infinite-dimensional state variable, whereas the bidirectional coupling prevents the use of decoupling techniques and PDE arguments available for classical FBSDEs.

The analysis presented here differs substantially from the existing literature. Owing to the delay, the forward dynamics become path-dependent and must be analyzed on an infinite-dimensional state space. Consequently, the associated decoupling field is characterized by an infinite-dimensional semilinear PDE. Furthermore, the strong interaction between the forward and backward equations prevents the direct use of contraction arguments available for weakly coupled systems. New stability estimates and perturbation arguments are therefore required.

The present paper provides the first asymptotic analysis of fully coupled delayed FBSDEs. Our contributions are threefold.
\begin{itemize}

\item [(i)] we derive quantitative stability estimates with respect to the perturbation parameter;
\item [(ii)] we establish convergence in distribution and almost sure convergence
\item [(iii)] combining these estimates with the weak convergence approach introduced by Cruzeiro et al.\cite{Cal}, we prove a large deviation principle for fully coupled delayed FBSDEs.
\end{itemize}

Consequently, the present paper closes the gap between the asymptotic theory of classical fully coupled FBSDEs developed by Cruzeiro et al. and the recent theory of delayed stochastic systems. It establishes the first large deviation principle for fully coupled forward-backward stochastic differential equations with delayed generators.

The remainder of the paper is organized as follows. Section 2 is devoted to the formulation of the problem and some preliminary results. Section 3 contains the main asymptotic analysis. More precisely, Subsection 3.1 is devoted to convergence in distribution, Subsection 3.2 establishes almost sure convergence, and Subsection 3.3 is concerned with the proof of the large deviation principle.

\section{Framework and preliminaries results}
	\subsection{Framework}
	Let $T>0$ be a fixed time horizon, and let $(\Omega,\mathcal{F},\P,(\mathcal{F}_t)_{0\leq t\leq T})$ be a filtered probability space satisfying the usual conditions, namely that the filtration $(\mathcal{F}_t)_{0\leq t\leq T}$ is complete, right-continuous, and generated by a one-dimensional Brownian motion $(W_t)_{0\leq t\leq T}$. We recall the following perturbed fully coupled delayed forward-backward stochastic differential equation (delayed FBSDE): for a given $\varepsilon>0$ and a fixed initial $(t,x)\in [0,T]\times\R$,
	\begin{eqnarray}\label{Eq1}
	\left\{
	\begin{array}{lll}
	X^{t,x,\varepsilon}(s)&=& \displaystyle x+\int^{s\vee t}_{t} b(r,X_r^{t,x,\varepsilon},Y_r^{t,x,\varepsilon},Z_r^{t,x,\varepsilon})dr+\sqrt{\varepsilon}\int^{s\vee t}_t \sigma(r,X_r^{t,x,\varepsilon},Y_r^{t,x,\varepsilon})dW(r) \\\\
	Y^{t,x,\varepsilon}(s)&=& \displaystyle g(X_T^{t,x,\varepsilon}) + \int_{s\vee t}^T f(r,X_r^{t,x,\varepsilon}, Y_r^{t,x,\varepsilon},Z_r^{t,x,\varepsilon})dr - \int_{s\vee t}^T Z^{t,x,\varepsilon}(r) d W(r),\;\; , s\in [0,T].	
	\end{array}
	\right.
	\end{eqnarray}
We also introduce the following delayed system of ordinary differential equations (ODEs): for a fixed $(t,x)\in [0,T]\times\R$,
\begin{eqnarray}\label{u}
\left\{
\begin{array}{lll}
\mathcal{X}^{t,x}(s)&=& \displaystyle x + \int_{t}^{t\vee s} b(s,\mathcal{X}^{t,x}_s, \mathcal{Y}^{t,x}_s,0)ds	\\\\
\mathcal{Y}^{t,x}(s)&=&\displaystyle g(\mathcal{X}^{t,x}_T) + \int_{t\vee s}^T f(r,\mathcal{X}^{t,x}_r, \mathcal{Y}^{t,x}_r,0)ds, \;\; s\in [0,T].
\end{array}
\right.
\end{eqnarray}

Throughout this paper, we shall work within the framework of the following topological vector spaces.
\begin{description}
\item $ \bullet $ Let $ L_{-T}^{\infty } (\mathbb{R} )$ denote the space of bounded, measurable functions $y: [-T,0] \rightarrow \mathbb{R} $\\
		satisfying
		$$
		\sup\limits_{-T\leq t\leq 0} \mid y(t) \mid^2 < +\infty,
		$$
\item $\bullet $ Let $L_{-T}^2 (\mathbb{R} ) $ denote the space of measurable functions $ z : [-T;0] \rightarrow \mathbb{R} $ satisfying
		$$   \int_{-T}^0 \mid z(t) \mid^2 dt < +\infty,
		$$		
\item $\bullet$ Let $ \mathcal{S}^2(\R)$ denote the space of all predictable process $\eta$ with values in $\R$ such that $$\E\left(\sup_{0\leq s\leq T}e^{\beta s}|\eta(s)|^2\right)<+\infty,$$
		\item $\bullet$ Let $\mathcal{H}^2(\R)$ denote the space of all predictable process $\eta$ with values in $\R$ such that $$\E\left(\int_{0}^Te^{\beta s}|\eta(s)|^2ds\right)<+\infty.$$
\end{description}
\begin{defi}
\begin{itemize}
\item [(a)] A triple of stochastic processes $(X^{t,x,\varepsilon},Y^{t,x,\varepsilon},Z^{t,x,\varepsilon})$ is said to be an adapted solution of the delayed FBSDE \eqref{Eq1} if, $(X^{t,x,\varepsilon},Y^{t,x,\varepsilon},Z^{t,x,\varepsilon})\in \mathcal{S}^2(\R)\times\mathcal{S}^{2}(\R)\times \mathcal{H}^2(\R)$and if equation \eqref{Eq1} is satisfied $\P$-almost surely.
\item [(b)] A pair of deterministic functions $(\mathcal{X}^{t,x},\mathcal{Y}^{t,x})$ is called a solution of the delayed system of ordinary differential equations \eqref{u} if, $(\mathcal{X}^{t,x},\mathcal{Y}^{t,x})\in L_{-T}^{\infty }\times L_{-T}^{\infty}$ and if system \eqref{u} is satisfied. 
\end{itemize}
\end{defi}
To conclude this subsection, we state the assumptions under which all the results of this paper will be established.
\begin{description}		
\item {(\bf A1)} $\phi: \Omega \times [-T,T]\times L_{-T}^{\infty}\times L_{-T}^{\infty}\times L_{-T}^2 (\mathbb{R}) \rightarrow \mathbb{R} $ be a product measurable and $ {\bf F} $-adapted function. Assume that there exist a probability measure $\alpha$ on $([-T,0],\mathcal{B}([-T,0]))$ and a positive constant $K>0$ such that, for $\P \otimes \lambda $-a.e. $(\omega ,t) \in \Omega \times[0,T]$ and for every $ u:= (\varphi,\psi,\eta); \, \ u':=(\varphi',\psi',\eta')$ in  $L_{-T}^{\infty} (\mathbb{R}) \times L_{-T}^{\infty} (\mathbb{R}) \times L_{-T}^2 (\mathbb{R})$ the following conditions hold:
\begin{itemize}
\item [(i)] (Delayed Lipschitz condition)
$\displaystyle
|\phi(t, u) - \phi(t,u')|^2 
\leq  K \int_{-T}^0  \Arrowvert u(r) - u'(r) \Arrowvert^2 \alpha (dr)$,
where $\Arrowvert u(r)\Arrowvert^2=|\varphi(r)|^2+|\psi(r)|^2+|\eta(r)|^2$.
\item [(ii)]For $t<0,\; \; \phi(t,u)= 0 $,
\item [(iii)] (Square integrability) $\displaystyle \mathbb{E} \left[ \int_{0}^T \vert \phi(t,0) \vert^2 dt\right] <  +\infty$.
\end{itemize}
These assumptions are imposed for $\phi=b, f$.
\item {(\bf A2)} $\sigma: \Omega \times [-T,T]\times L_{-T}^{\infty}\times L_{-T}^{\infty}\times L_{-T}^2 (\mathbb{R}) \rightarrow \mathbb{R}$ be a product measurable and $ {\bf F} $-adapted function. Assume that there exist a probability measure $\alpha$ on $([-T,0],\mathcal{B}([-T,0]))$ and a positive constant $K>0$ such that, for $\P \otimes \lambda $-a.e. $(\omega ,t) \in \Omega \times[0,T]$ and for every $w := (\varphi,\psi); \, \ w':=(\varphi',\psi')$ in  $L_{-T}^{\infty} (\mathbb{R}) \times L_{-T}^{\infty} (\mathbb{R})$ the following conditions hold:
\begin{itemize}
\item [(i)] (Delayed Lipschitz condition)
$\displaystyle
|\sigma(t, w) - \sigma(t,w')|^2 
\leq  K \int_{-T}^0  \Arrowvert w(r) - w'(r) \Arrowvert^2 \alpha (dr)$,
where $\Arrowvert w(r)\Arrowvert^2=|\varphi(r)|^2+|\psi(r)|^2$.
\item [(ii)]For $t<0,\; \; \sigma(t,w)= 0 $,
\item [(iii)] (Square integrability) $\displaystyle \mathbb{E} \left[ \int_{0}^T \vert \sigma(t,0) \vert^2 dt\right] <  +\infty$.
\end{itemize}		
\item {(\bf A3)} Let $ g: \Omega \times [-T,T]\times L_{-T}^2 (\mathbb{R}) \rightarrow \mathbb{R} $ be a product measurable and $ {\bf F} $-adapted function. Assume that there exist a probability measure $\alpha$ on $([-T,0],\mathcal{B}([-T,0]))$ and a positive constant $K>0$ such that
		$$\displaystyle |g(\varphi)- g(\varphi')|^2\leq K\int_{-T}^{0}|\varphi(r)-\varphi(r)|^2\alpha(dr).$$
\item {(\bf A4)} The functions $\sigma$ et $g$ satisfy also the following assumptions. There exist three constants $C,\lambda, \gamma$ such that 
\begin{itemize}
\item [(i)] $\forall\; t\in [0,T],\; (\varphi, \psi, \eta)\in L_{-T}^{\infty} (\mathbb{R}) \times L_{-T}^{\infty} (\mathbb{R})$, 
\begin{eqnarray*}
	|\sigma(t,\varphi,\psi)|+|g(\varphi)|\leq C.
\end{eqnarray*}
\item [(ii)] $\forall\; t\in [0,T],\; (\varphi, \psi)\in L_{-T}^{\infty} (\mathbb{R}) \times L_{-T}^{\infty} (\mathbb{R})$
\begin{eqnarray*}
	\langle\xi, a(t,\varphi, \psi)\xi\rangle \geq \lambda |\xi|^2, \;\; \forall\; \xi\in \R^{n}, 
\end{eqnarray*}
where $a(t,\varphi, \psi)=\sigma\sigma^{*}(t,\varphi,\psi)$.
\item [(iii)] The function $\sigma$ is differentiable with respect to $\psi$ and $\eta$ and its derivatives with respect to $\varphi$ and $\psi$ are $\gamma$-hölder in $\varphi$ and $\psi$, uniformly in $t$ and $\eta$.
\end{itemize}
\end{description}
\begin{rem}\label{R1}
\begin{itemize}
\item [(a)] Assumption $(\bf A1)$-$(ii)$ and $(\bf A2)$-$(ii)$ allows us to extend the solutions of \eqref{Eq1} and \eqref{u} to negative times. More precisely, for every $s<0$, we set $(X^{t,x,\varepsilon}(s),Y^{t,x,\varepsilon}(s),Z^{t,x,\varepsilon}(s))=(x,Y^{t,x,\varepsilon}(t),0)$ and $(\mathcal{X}^{t,x}(s),\mathcal{Y}^{t,x}(s),0)=(x,\mathcal{Y}^{t,x}(t),0)$, which are consistent with equations \eqref{Eq1} and \eqref{u}, respectively.
\item [(b)] The notation $\phi(t,{\bf 0})$ appearing in condition $({\bf A1})$-$(iii)$ denotes the value of the coefficient $\phi$ evaluated at the null trajectory, namely $u_t=(0,0,0)$.
\item [(c)] The differentiability of the function $\sigma$ is defined as follows:  $\sigma$ is Fréchet differentiable at $(\varphi,\psi)\in L_{-T}^{\infty } (\mathbb{R}) \times L_{-T}^{\infty} (\mathbb{R})$ if there exists a continuous linear mapping $D\sigma(t,\varphi,\psi): L_{-T}^{\infty}(\mathbb{R})\times L_{-T}^{\infty} (\mathbb{R})\rightarrow \R$ such that
\begin{eqnarray*}
\sigma(t,\varphi+h,\psi+k)-\sigma(t,\varphi,\psi)=D\sigma(t,\varphi,\psi)(h,k)+o(\|(h,k)\|).
\end{eqnarray*}
for all $t\in [0,T]$.
\end{itemize}
\end{rem}
We conclude this subsection by stating the existence and uniqueness results for equations \eqref{Eq1} and \eqref{u}.    
\begin{propo}\label{Prop1}
Assume that conditions $({\bf A1})$-$({\bf A3})$ are satisfied. Moreover, suppose that either the time horizon $T$ or the Lipschitz constant $K$ is sufficiently small so that 
\begin{eqnarray}\label{c1}
		(25+ 144K)Ke \max (1,T^2) < 1.
\end{eqnarray}
Then, the delayed forward-backward stochastic differential equation \eqref{Eq1} admits a unique adapted solution $(X^{t,x,\varepsilon}, Y^{t,x,\varepsilon}, Z^{t,x,\varepsilon} )$. 
\end{propo}
\begin{proof}
The proof follows essentially the same arguments as those used in Theorem 3.1 of \cite{AHJ}. The main difference arises from the form of the terminal condition in the backward equation. In the present framework, the terminal condition is given by $g(X^{t,x,\varepsilon}_T)$ whereas in \cite{AHJ} it is represented by a square-integrable random variable $\xi$. This distinction significantly affects the corresponding a priori estimates and explains why the conditions ensuring existence and uniqueness are different in the two settings. More precisely, in the present paper the required condition is
\begin{eqnarray*}
	(25+ 144K)Ke\max(1,T^2)< 1,
\end{eqnarray*}
while in \cite{AHJ} the corresponding condition is
\begin{eqnarray*}
	21Ke\max(1,T^2)< 1.
\end{eqnarray*}
\end{proof}

\begin{propo}
Assume $({\bf A1})$-$({\bf A3})$ are satisfied. Then, the delayed system of ordinary differential equations \eqref{u} admits a unique solution.
\end{propo}

\subsection{Preliminaries results}
This subsection is devoted to a collection of preliminary estimates that will be repeatedly used in the sequel. The following four lemmas constitute the main technical tools underlying the proofs of our principal results.
	
\begin{lem}\label{l1}
Suppose that assumptions $({\bf A1})$--$({\bf A3})$ hold. Furthermore, assume that either the time horizon $T$ or the Lipschitz constant $K$ is sufficiently small so that
\begin{eqnarray}\label{l}
(57 + 432K)K e \max(1, T^2)< 1.
\end{eqnarray}			
Then there exists a positive constant $C$, independent of $\varepsilon$, such that for every $x,y\in\R$,
\begin{eqnarray*}
&&\E\left[\sup_{ t\leq s \leq T}e^{\beta s} |X^{t,x,\varepsilon}(s)- X^{t,y,\varepsilon}(s)|^2 +  \sup_{ t\leq s \leq T}e^{\beta s} |Y^{t,x,\varepsilon}(s)- Y^{t, y,\varepsilon}(s)|^2\right.\nonumber\\
&&\left.+ \int_t^T e^{\beta s} |Z^{t,x,\varepsilon}(s)- Z^{t, y,\varepsilon}(s)|^2 ds \right]\leq  C |x-y|^2.
\end{eqnarray*}
\end{lem}
\begin{proof}
Let set 
\begin{eqnarray*}
\bar{H}^{t, x,y,\varepsilon} &=& H^{t, x,\varepsilon} - H^{t, y,\varepsilon},\\
\bar{\varphi}(s) &=& \varphi(s, X_s^{t,x,\varepsilon},Y_s^{t,x,\varepsilon}, Z_s^{t,x,\varepsilon})- \varphi(s, X_s^{t,y,\varepsilon},Y_s^{t,y,\varepsilon},Z_s^{t,y,\varepsilon})
\end{eqnarray*}
for  $H= X, Y, Z$ and $\varphi:= b, f, \sigma, g$. It follows from FBSDE \eqref{Eq1} that
\begin{eqnarray*}
\left\{
\begin{array}{lll}
\bar{X}^{t, x,y,\varepsilon}(s)&	=&\displaystyle x-y + \int_t^s \bar{b}(r) dr + \int_t^s \bar{\sigma}(r) dW(r)	\\\\
\bar{Y}^{t, x,y,\varepsilon}(s)& = & \displaystyle\bar{g}(X_T^{ t}) + \int_s^T\bar{f}(r) dr - \int_s^T \bar{Z}^{t, x,y,\varepsilon}(r)dW(r),
\end{array}\right.
\end{eqnarray*}
so that applying Itô's formula, taking expectation with respect to $\mathcal{F}_t$, Cauchy-Schwartz inequality, isometry property and since $\varepsilon$ smaller then $1$, we get 

\begin{eqnarray}\label{x0}
\E \left[ \sup_{0 \leq s \leq T} e^{\beta s} | \bar{X}^{t, x,y,\varepsilon}(s) |^2 \right] &\leq & \left(1- \frac{3}{4}\beta^2 T^2\right)^{-1} 3|x- y|^2\nonumber \\ &&  +  \left(1- \frac{3}{4}\beta^2 T^2\right)^{-1} 6\max(1, T)\E \left[\int_0^T e^{\beta s}|\bar b (s)|^2  ds +  \int_0^T e^{\beta s}|\bar \sigma (s)|^2  ds \right]\nonumber\\
\end{eqnarray}	
and 
\begin{eqnarray}\label{yz0}
\E\left[\sup_{ t\leq s \leq T} e^{\beta s}|\bar Y^{t,x,\varepsilon} (s)|^2 + \int_t^T e^{\beta s}|\bar Z^{ x,y} (s) |^2 ds\right]&\leq&\E \left(9 e^{\beta T}|\bar g(T)|^2 +\left(8T +\frac{1}{\beta}\right)\int_t^T e^{\beta s} |\bar f(s)|^2 ds\right).\nonumber\\	
\end{eqnarray}	
Taking $\beta= \frac{1}{T} $ on the equations \eqref{x0}, \eqref{yz0} together with Assumptions $({\bf A1})$ and $({\bf A2})$ and \eqref{l} we get
\begin{eqnarray*}
\E\left[\sup_{ 0\leq s \leq T}e^{\beta s} |\bar{X}^{t,x,\varepsilon}(s)|^2 +  \sup_{ 0\leq s \leq T}e^{\beta s} |\bar{Y}^{t,x,\varepsilon}(s)|^2+ \int_t^T e^{\beta s} |\bar{Z}^{t,x,\varepsilon}(s)|^2 ds \right]& \leq & C|x-y|^2
\end{eqnarray*}
\end{proof}

\begin{lem}\label{l2}
Suppose that assumptions $({\bf A1})$--$({\bf A3})$ hold. Furthermore, assume that either the time horizon $T$ or the Lipschitz constant $K$ is sufficiently small so that
\begin{eqnarray}\label{c2}
(66+ 864K)K e \max(1,T^2) < 1.
\end{eqnarray}
Then there exists a positive constant $C$, independent of $\varepsilon$, such that
\begin{eqnarray*}
\E\left[ \sup_{ 0\leq s \leq T}e^{\beta s}|X^{t,x,\varepsilon}(s) |^2 + \sup_{ 0\leq s\leq T}e^{\beta s} |Y^{t,x,\varepsilon}(s) |^2  + \int_{ 0}^Te^{\beta s} |Z^{t,x,\varepsilon}(s) |^2ds\right]  \leq C \left(1+ |x|^2\right).   
\end{eqnarray*}	
\end{lem}
\begin{proof}
Using arguments similar to those developed above, we derive the estimate
\begin{eqnarray}\label{x1}
\E\left[\sup_{ 0\leq s \leq T}e^{\beta s}	|X^{t,x,\varepsilon}(s)|^2 \right] 
&\leq &   3\left(1- \frac{3}{4}\beta^2 T^2\right)^{-1}|x|^3\\
&&+ 6\left(1- \frac{3}{4}\beta^2 T^2\right)^{-1}\max(1, T)\E \left[  \int_0^T e^{\beta s}|b (s)|^2  ds +  \int_0^T e^{\beta s}|\sigma (s)|^2  ds. \right]\nonumber
\end{eqnarray} 
Moreover, we obtain
\begin{eqnarray}\label{yz}
&&\E\left[\sup_{ 0\leq t \leq T} e^{\beta t}|Y^{t,x,\varepsilon} (t)|^2 + \int_0^T e^{\beta s}|Z^{t,x,\varepsilon} (s) |^2 ds\right]\nonumber\\ 
&\leq & \E \left(9 e^{\beta T}|g(X^{t,x,\varepsilon}_T)|^2+(8T +\frac{1}{\beta})\int_0^T e^{\beta s} |f(s)|^2 ds\right).
\end{eqnarray}
Finally, choosing $\beta= \dfrac{1}{T} $ in \eqref{x1} and \eqref{yz}, and combining these estimates with Assumptions  $({\bf A1})$ and $({\bf A2})$, and condition \eqref{l}, we conclude that there exists a positive constant $C$, independent of $\varepsilon $, such that
\begin{eqnarray*}
\E\left[\sup_{ 0\leq s \leq T}e^{\beta s} |X^{t,x,\varepsilon}(s)|^2 +  \sup_{ 0\leq s \leq T}e^{\beta s} |Y^{t,x,\varepsilon}(s)|^2+ \int_t^T e^{\beta s} |Z^{t,x,\varepsilon}(s)|^2 ds \right]& \leq & C(1+|x|)^2	.	
\end{eqnarray*}
\end{proof}

\begin{rem}
It is worth noting that if the terminal condition $g$ is bounded, then condition \eqref{c2} can be relaxed to
\begin{eqnarray*}
66 K e  \max(1, T^2)< 1.	
\end{eqnarray*}
\end{rem}
The following result provides a quantitative continuity estimate for the solution with respect to the initial time parameter. 
\begin{lem}\label{l3}
Suppose that assumptions $({\bf A1})$--$({\bf A3})$ hold. Furthermore, assume that the constants $T$ and $K$ are sufficiently small so that
\begin{eqnarray}\label{e3}
(57+432K)K e \max(1,T)<1.
\end{eqnarray}
Then there exists a positive constant $C$, independent of $\varepsilon$, such that for every $t,t'\in[0,T]$,
\begin{eqnarray*}
&&\E\left[\sup_{ t\leq s \leq T}e^{\beta s} |X^{t,x,\varepsilon}(s)- X^{t',x,\varepsilon}(s)|^2 +  \sup_{ t\leq s \leq T}e^{\beta s} |Y^{t,x,\varepsilon}(s)- Y^{t', x,\varepsilon}(s)|^2\right.\nonumber\\
&&\left.+ \int_t^T e^{\beta s} |Z^{t,x,\varepsilon}(s)- Z^{t', x,\varepsilon}(s)|^2 ds \right] \leq  C|t-t'|^2.
\end{eqnarray*}
\end{lem}
\begin{proof}
For notational convenience, define 
\begin{eqnarray*}
\overline{H}^{t, t',x,\varepsilon} &=& H^{t, x,\varepsilon} - H^{t', x,\varepsilon}, \qquad H\in\{X,Y,Z\},
\end{eqnarray*}
\begin{eqnarray*}
\bar{\varphi}(s) &=& \varphi(s, X_s^{t,x,\varepsilon},Y_s^{t,x,\varepsilon}, Z_s^{t,x,\varepsilon})- \varphi(s, X_s^{t',x,\varepsilon},Y_s^{t',x,\varepsilon},Z_s^{t',x,\varepsilon}),
\end{eqnarray*}
and 
\begin{eqnarray*}
\psi(s)&=& \psi(s, X_s^{t, x,\varepsilon},Y_s^{ t, x;\varepsilon},Z_s^{t, x,\varepsilon})
\end{eqnarray*}
for $\varphi\in\{b,\sigma,f,g\}$ and $\psi\in\{b,\sigma\}$.
By decompose $e^{\frac{\beta}{2}s}|\overline{X}^{t,t',x,\varepsilon}(s)|$ according to the partition $[0,T]=[0,t]\cup[t,t']\cup[t',T]$, we get 
\begin{eqnarray*}
e^{\frac{\beta}{2}s}|\overline{X}^{t, t', x,\varepsilon}(s)| &\leq & e^{\frac{\beta}{2}s}|\overline{X}^{t, t', x,\varepsilon}(s)| {\bf 1}_{[0, t]}(s) + e^{\frac{\beta}{2}s}|\overline{X}^{t, t', x,\varepsilon}(s)| {\bf 1}_{[t, t']}(s)
			+  e^{\frac{\beta}{2}s}|\overline{X}^{t, t', x,\varepsilon}(s)| {\bf 1}_{[t', T]}(s).
\end{eqnarray*}
We first observe that
\begin{eqnarray*}
e^{\frac{\beta}{2}s}|\overline{X}^{t, t', x,\varepsilon}(s)|{\bf 1}_{[0, t]}(s)=0.
\end{eqnarray*}
Next, for $s\in[t,t']$,
\begin{eqnarray*}
 e^{\frac{\beta}{2}s}|\overline{X}^{t, t', x,\varepsilon}(s)| {\bf 1}_{[t, t']}(s)= e^{\frac{\beta}{2}s}| X^{t', x,\varepsilon}(s) - x|{\bf 1}_{[t, t']}(s)	
\end{eqnarray*}
so that with Itô's formula, expectation with respect $\mathcal{F}_{t}$, Hölder inequality  and isometry, we get
\begin{eqnarray*}
\E\left[\sup_{t \leq s \leq t'}e^{\beta s}|X^{t,x,\varepsilon}(s) - x|^2 \right] &\leq & 4 \left(1-\frac{\beta^2}{2}|t'-t|^2\right)^{-1} \E \left[|t'-t| \int_{t}^{t'}e^{\beta r}|b(r)|^2 dr+ 2 \int_{t}^{t'}e^{\beta r}|\sigma(r)|^2 dr\right].
\end{eqnarray*} 
By using the lemma \eqref{l2} and taking  $\beta= \frac{1}{t_1- t_2}$, we obtain
\begin{eqnarray} \label{t2t1}
\E\left[\sup_{t\leq s \leq t'}e^{\beta s}|X^{t,x,\varepsilon}(s) - x|^2 \right] &\leq & C(1+ |x|) |t'-t|^2. 
\end{eqnarray} 
Finally, for $s\in[t',T]$, with the same arguments we obtain	
\begin{eqnarray*}
e^{\frac{\beta}{2}s}\overline{X}^{t, t', x,\varepsilon}(s)&=& e^{\frac{\beta}{2}t'} [X^{t,x,\varepsilon}(t')- x] + \frac{\beta}{2}\int_{t'}^{s}	e^{\frac{\beta}{2}r}\overline{X}^{t, t', x,\varepsilon}(r)dr\\ 
&&+\int_{t'}^{s}	e^{\frac{\beta}{2}r}\bar{b}(r)dr +\int_{t'}^{s}	e^{\frac{\beta}{2}r}\bar{\sigma}(r) dW(r) .	
\end{eqnarray*}
and 	
\begin{eqnarray*}
\E\left[ \sup_{t' \leq s \leq T}	e^{\beta s} |\overline{X}^{t, t', x,\varepsilon}(s)|^2  \right] &\leq & 3(1-\frac{3}{4}\beta^2T^2 )^{-1} \E  [e^{\beta t'}| X^{t,x}(t')- x|^2 ]\\
&& + 6 (1-\frac{3}{4}\beta^2T^2 )^{-1} \E \left[T \int_{t'}^{s} e^{\beta r} |\bar{b}(r)|^2 dr + \int_{t'}^{s} e^{\beta r} |\bar{\sigma}(r)|^2 dr \right].
\end{eqnarray*}
Taking $\beta= \frac{1}{T}$ and the relation \eqref{t2t1} we have
\begin{eqnarray}\label{x13} 
\E\left[\sup_{t' \leq s \leq T}e^{\beta s}|\overline{X}^{t, t', x,\varepsilon}(s) |^2 \right] &\leq & 12 C(1+ |x|)|t'- t|^2 + 24  \max(1, T) \E \left[ \int_{0}^{T} e^{\beta r} |\bar{b}(r)|^2 dr + \int_{0}^{T} e^{\beta r} |\bar{\sigma}(r)|^2 dr \right].\nonumber\\
\end{eqnarray}

On the other hand, applying Itô's formula to $e^{\frac{\beta s}{2}} \overline{Y}^{t, t', x,\varepsilon}(s)$, we get
\begin{eqnarray*}
e^{\frac{\beta}{2}s}\overline{Y}^{t, t', x,\varepsilon}(s)&=&		e^{\frac{\beta}{2}T}\bar{g}(T) + \int_{s}^T 	e^{\frac{\beta}{2}r} \bar f(r) dr - \int_{s}^T e^{\frac{\beta}{2}r}\overline{Z}^{t, t', x,\varepsilon}(r) dW(r).  
\end{eqnarray*}
Taking expectation with respect $\mathcal{F}_s$ and Doob's inequality we have
\begin{eqnarray}\label{y3}
\E \left[\sup_{t \leq s \leq T} e^{\beta s}|\overline{Y}^{t, t', x,\varepsilon}(s)|^2 \right] \leq 8 \E \left[e^{\beta T}|\bar{g}(T)|^2 + T \int_0^T e^{\beta r } |\bar{f}(r) |^2 dr \right].
\end{eqnarray}
Applying again Itô's formula to $e^{\beta s}| \overline{Y}^{t, t', x,\varepsilon}(s)|^2$ together with Young's inequality and take the expectation we obtain
\begin{eqnarray*}
\E	\left[ e^{\beta s}|\overline{Y}^{t, t', x,\varepsilon}(s)|^2 + \int_{s}^{T} e^{\beta r}|\overline{Z}^{t, t', x,\varepsilon}(r)|^2 dr \right] \leq \E \left[ e^{\beta T} |\bar{g}(T)|^2 +  \frac{1}{\beta} \int_s^T e^{\beta r} | \bar{f}(r)|^2 dr\right].
\end{eqnarray*}	
Setting $s=0$ yields
\begin{eqnarray}\label{z3}
\E\left[  \int_{0}^{T} e^{\beta r}|\overline{Z}^{t, t', x,\varepsilon}(r)|^2 dr \right] \leq \E \left[ e^{\beta T} |\bar{g}(T)|^2 +  \frac{1}{\beta} \int_0^T e^{\beta r} | \bar{f}(r)|^2 dr\right].
\end{eqnarray}	
Combining \eqref{z3} and \eqref{y3} and taking $\beta= \frac{1}{T}$, we obtain
\begin{eqnarray}\label{yz3}
\E	\Big[\sup_{t \leq s \leq T} e^{\beta s}|\overline{Y}^{t, t', x,\varepsilon}(s)|^2  + \int_{0}^{T} e^{\beta r}|\overline{Z}^{t, t', x,\varepsilon}(r)|^2 dr\Big] &\leq & 9\E [e^{\beta T} \bar{g}(T) ]+ 9T \E  \left[\int_0^T e^{\beta r} | \bar{f}(r)|^2 dr\right] \nonumber \\
 & \leq & 9 K \E\left( \sup_{t_1 \leq s \leq T} e^{\beta s}|\overline{X}^{t, t', x,\varepsilon}(s) |^2\right)\nonumber\\
 &&+ 9 \max(1, T) \E\left[   \int_0^T e^{\beta r}|\bar{f}(r)|^2 dr\right]. 
\end{eqnarray}
In view of assumption $({\bf A1})$-$({\bf A3})$, we have
\begin{eqnarray*}
\E \left[\int_0^T e^{\beta r} | \bar{b}(r)|^2 dr\right] &\leq & Ke \E\left[\sup_{t \leq s \leq T} e^{\beta s}|\overline{X}^{t, t', x,\varepsilon}(s)|^2 + \sup_{t' \leq s \leq T} e^{\beta s}|\overline{Y}^{t, t', x,\varepsilon}(s) |^2\right], 
\end{eqnarray*}

\begin{eqnarray*}
\E\left[\int_0^T e^{\beta r} | \bar{\sigma}(r)|^2 dr\right] &\leq & K e \E \left[\sup_{t \leq s \leq T} e^{\beta s}|\overline{X}^{t, t', x,\varepsilon}(s)|^2 + \sup_{t \leq s \leq T} e^{\beta s}|\overline{Y}^{t, t', x,\varepsilon}(s) |^2\right]
\end{eqnarray*}
and 
\begin{eqnarray*}
\E\left[ \int_0^T e^{\beta r} | \bar{f}(r)|^2 dr\right] &\leq & K e \E \left[\sup_{t \leq s \leq T} e^{\beta s}|\overline{X}^{t, t', x,\varepsilon}(s) |^2 + \sup_{t \leq s \leq T} e^{\beta s}|\overline{Y}^{t, t', x,\varepsilon}(s)|^2 + \int_0^T e^{\beta s}|\overline{Z}^{t, t', x,\varepsilon}(s) |^2 ds\right].
\end{eqnarray*}
It follows from \eqref{t2t1}, \eqref{x13} and \eqref{yz3} that   
\begin{eqnarray*}
&&\E \left[\sup_{t \leq s \leq T} e^{\beta s}|\overline{X}^{t, t', x,\varepsilon}(s) |^2 + \sup_{t \leq s \leq T} e^{\beta s}|\overline{Y}^{t, t', x,\varepsilon}(s) |^2 + \int_0^T e^{\beta s}|\overline{Z}^{t, t', x,\varepsilon}(s) |^2 ds\right]\\
&\leq & 12 (1+ 9K) C(1+ |x|) |t' - t|^2\\
&& + (48K+432K^2)e \max(1, T)\E\left[\sup_{t \leq s \leq T} e^{\beta s}|\overline{X}^{t, t', x,\varepsilon}(s) |^2 + \sup_{t \leq s \leq T} e^{\beta s}|\overline{Y}^{t, t', x,\varepsilon}(s) |^2\right]\\
  &&+ 9K e \max (1, T) \E \left[\sup_{t \leq s \leq T} e^{\beta s}|\overline{X}^{t, t', x,\varepsilon}(s) |^2 + \sup_{t \leq s \leq T} e^{\beta s}|\overline{Y}^{t, t', x,\varepsilon}(s) |^2 + \int_0^T e^{\beta s}|\overline{Z}^{t, t', x,\varepsilon}(s) |^2 ds\right]\\
  &\leq &  12 (1+ 9K) C(1+ |x|) |t' - t|^2\\
  && +  	(57+ 432K)K e \max(1,T) \E \left[\sup_{t \leq s \leq T} e^{\beta s}|\overline{X}^{t, t', x,\varepsilon}(s) |^2 + \sup_{t \leq s \leq T} e^{\beta s}|\overline{Y}^{t, t', x,\varepsilon}(s) |^2 + \int_0^T e^{\beta s}|\overline{Z}^{t, t', x,\varepsilon}(s) |^2 ds\right]
  \end{eqnarray*}
Invoking assumption \eqref{e3} and absorbing the last term into the left-hand side, we deduce that 
\begin{eqnarray*}
\E \left[\sup_{t \leq s \leq T} e^{\beta s}|\overline{X}^{t, t', x,\varepsilon}(s) |^2 + \sup_{t \leq s \leq T} e^{\beta s}|\overline{Y}^{t, t', x,\varepsilon}(s) |^2 + \int_0^T e^{\beta s}|\overline{Z}^{t, t', x,\varepsilon}(s) |^2 ds\right] \leq C|t' - t|^2.
\end{eqnarray*}
This completes the proof.
\end{proof}
The following lemma provides a stability estimate with respect to the perturbation parameter $\varepsilon$.
\begin{lem}\label{l4}
Suppose that assumptions $({\bf A1})$-$({\bf A3})$ hold. Let For $0< \varepsilon < \varepsilon' < 1$ and and assume that the constants $T$ and $K$ are sufficiently small so that
\begin{eqnarray}\label{cc}
(33+ 144K)K e \max(1, T^2) < 1.
\end{eqnarray}
Then there exists a positive constant $C$, independent of $\varepsilon_1$, $\varepsilon_2$, and $T$, such that 
\begin{eqnarray*}
&&\E \left[\sup_{t \leq s\leq T} e^{\beta t} | X^{t,x,\varepsilon'} (s) - X^{t,x,\varepsilon}(s)|^2 + \sup_{t \leq s \leq T}e^{\beta t} |Y^{t,x,\varepsilon'} (s) - Y^{t,x,\varepsilon}(s)|^2\right.\\
&& \left.+ \int_0^Te^{\beta s}| Z^{t,x,\varepsilon'} (s) - Z^{t,x,\varepsilon}(s)|^2 ds\right]
\leq  C (\sqrt{\varepsilon'} - \sqrt{\varepsilon})^2 	
\end{eqnarray*}
\end{lem}
\begin{rem}
As a direct consequence of Lemma \ref{l4}, we have
\begin{eqnarray*}
&&\E\left[\sup_{t \leq s\leq T} e^{\beta t} | X^{t,x,\varepsilon'} (s) - X^{t,x,\varepsilon}(s)|^2 + \sup_{t \leq s \leq T}e^{\beta t} |Y^{t,x,\varepsilon'} (s) - Y^{t,x,\varepsilon}(s)|^2\right.\\
&& \left.+ \int_0^Te^{\beta s}| Z^{t,x,\varepsilon'} (s) - Z^{t,x,\varepsilon}(s)|^2 ds\right]\leq C\,|\varepsilon'-\varepsilon|.
\end{eqnarray*}
Hence, the family of solutions
$\{(X^{t,x,\varepsilon},Y^{t,x,\varepsilon},Z^{t,x,\varepsilon})\}_{\varepsilon>0}$
depends continuously on the perturbation parameter $\varepsilon$ in the corresponding weighted $L^2$-norm.
\end{rem}
\begin{proof}
Throughout the proof, we use the following notation: 
\begin{eqnarray*}
\overline{H}^{t,x,\varepsilon,\varepsilon'} &=& H^{t, x,\varepsilon'} - H^{t, x,\varepsilon}, \qquad H\in\{X,Y,Z\},
\end{eqnarray*}
\begin{eqnarray*}
\bar{\varphi}(s) &=& \varphi(s, X_s^{t,x,\varepsilon'},Y_s^{t,x,\varepsilon'}, Z_s^{t,x,\varepsilon'})- \varphi(s, X_s^{t,x,\varepsilon},Y_s^{t,x,\varepsilon},Z_s^{t,x,\varepsilon}),
\end{eqnarray*}
and 
\begin{eqnarray*}
\sigma(s)&=& \sigma(s, X_s^{t, x,\varepsilon},Y_s^{ t, x;\varepsilon},Z_s^{t, x,\varepsilon})
\end{eqnarray*}
for $\varphi\in\{b,\sigma,f,g\}$.

Next, for all $s\in [t,T]$, we have
\begin{eqnarray*}
\overline{X}^{t,x,\varepsilon,\varepsilon'}(s)& = & \int_t^s\bar{b}(s)ds + \sqrt{\varepsilon'}\int_t^s \bar{\sigma}(s)\,dW(s)+(\sqrt{\varepsilon'}- \sqrt{\varepsilon})\int_t^s  \sigma(s)\,dW(s).
\end{eqnarray*}
Applying Itô's formula to $e^{\frac{\beta}{2}s}   |	\overline{X}^{\epsilon}(s)|$ and then taking conditional expectation which respect $\mathcal{F}_s$  we have
\begin{eqnarray*}
			e^{\frac{\beta}{2}s}   |	\overline{X}^{t,x,\varepsilon,\varepsilon'}(s)| & \leq & \E \left[  \frac{\beta}{2} \int_t^T e^{\frac{\beta}{2}r}   |\overline{X}^{t,x,\varepsilon,\varepsilon'}(r)|dr + \int_t^T e^{\frac{\beta}{2}s}|  \bar{b}(r)| dr  + \sqrt{\epsilon'} \int_0^T \bar{\sigma}(r) dW(r)\right.\\
			 &&\left.+ (\sqrt{\epsilon'} - \sqrt{\epsilon}) \int_t^T \sigma_2(r)dW(r) | \mathcal{F}_s\right].
\end{eqnarray*}
Using Cauchy-Schwarz's inequality, Itô's isometry, and the fact that $\varepsilon'<1$,
we obtain
\begin{eqnarray*}
\E \left[ \sup_{t \leq s \leq T} e^{\beta t}|\overline{X}^{t,x,\varepsilon,\varepsilon'}(s)|^2 \right] & \leq &  \frac{\beta^2}{2}T    	\int_t^T\E( e^{\beta s}|\overline{X}^{t,x,\varepsilon,\varepsilon'}(s)|^2) ds\\ 
&& + 6 \E \left( T \int_t^T e^{\beta s} | \bar b(s)|^2 ds + \varepsilon'\int_t^T e^{\beta s} |\bar \sigma(s)|^2 ds    \right. \\
&& \left. + (\sqrt{\varepsilon'}- \sqrt{\varepsilon})^2\int_t^T e^{\beta s} | \sigma_2(s)|^2 ds \right)\\ 
& \leq & \frac{\beta^2}{2}T^2  \E [\sup_{t \leq s \leq T} e^{\beta t}|\overline{X}^{\epsilon}(t)|^2 ]\notag \\
&& + 6 \max( 1, T) \E \left[ \int_t^T e^{\beta s}| \bar b(s)|^2 ds + \int_0^T e^{\beta s}| \bar \sigma(s)|^2 ds  \right]\notag \\ 
&& + 6 (\sqrt{\varepsilon'}- \sqrt{\varepsilon})^2 \E \left( \int_0^T e^{\beta s} | \sigma_2(s)|^2 ds \right).
\end{eqnarray*}
Then
\begin{eqnarray}\label{x2}
\E \left[ \sup_{t \leq s \leq T} e^{\beta s}|\overline{X}^{t,x,\varepsilon,\varepsilon'}(s)|^2   \right] & \leq & 6\left( 1-\frac{\beta^2}{2}T^2   \right)^{-1}  \max( 1, T) \E \left[ \int_t^T e^{\beta s}| \bar b(s)|^2 ds + \int_t^T e^{\beta s}| \bar \sigma(s)|^2 ds  \right]\nonumber \\
&& + 6\left( 1-	\frac{\beta^2}{2}T^2   \right)^{-1} \Big(\sqrt{\epsilon'}- \sqrt{\varepsilon'}\Big)^2 \E \left( \int_t^T e^{\beta s} | \sigma_2(s)|^2 ds \right).
\end{eqnarray}
It follows from \eqref{x2} and \eqref{yz} respectively that
\begin{eqnarray}\label{zy}
&&\E \left[\sup_{0 \leq t \leq T} e^{\beta t} | \overline{Y}^{t,x,\varepsilon,\varepsilon'}(s)|^2 + \int_t^T e^{\beta s} |\overline{Z}^{t,x,\varepsilon,\varepsilon'}(s)|^2 ds \right]\nonumber\\
 &\leq & \E \left[9 e^{\beta T} | \overline g(\overline{X}^{t,x,\varepsilon,\varepsilon'}_T)|^2 + (8T+ \frac{1}{\beta})\int_0^T e^{\beta s} |\bar f (s)|^2 ds \right]\nonumber \\
&\leq & 108 K \left( 1-	\frac{\beta^2}{2}T^2\right)^{-1}  \max(1, T) \E \left[ \int_t^T e^{\beta s}| \bar b(s)|^2 ds + \int_0^T e^{\beta s}| \bar \sigma(s)|^2 ds  \right]\nonumber\\
&& + 108 K \left( 1-	\frac{\beta^2}{2}T^2   \right)^{-1} \Big(\sqrt{\epsilon'}- \sqrt{\epsilon}\Big)^2 \E \left[ \int_t^T e^{\beta s} | \sigma(s)|^2 ds\right.\nonumber\\
&&\left.+ (8T+ \frac{1}{\beta})\int_0^T e^{\beta s} |\bar f (s)|^2 ds \right]
\end{eqnarray}
Using assumptions $({\bf A1})$-$({\bf A3})$ and taking $\beta=\frac{1}{T}$,  we get 
\begin{eqnarray*}
\E\left(\int_0^T e^{\beta s} |\bar{\varphi}|^2 ds\right) &\leq &  K \E \left[\int_0^T e^{\beta s} \int_{-T}^0 ( |\bar X^\epsilon(s+ u)|^2 + |\bar Y^\epsilon(s+ u)|^2 +|\bar Z^\epsilon(s+ u)|^2   ) \alpha(du) ds \right] \nonumber\\
 &\leq & K e \max (1,T) \E \left[\sup_{t \leq s \leq T} e^{\beta s}|\overline X^{\epsilon}(s) |^2     + \sup_{t \leq s \leq T} e^{\beta s}|\overline Y^{\epsilon}(s) |^2   + \int_t^T  e^{\beta s}  |\overline Z^{\epsilon}(s)|^2 ds\right].
\end{eqnarray*}
On the other hand, using assumption ({\bf A1}) together with Lemma \ref{l2}, we obtain
\begin{eqnarray*}
\E \left[ \int_t^T e^{\beta s} | \sigma(s)|^2 ds\right]\leq C,	
\end{eqnarray*} 
where $C$ is independent of $\varepsilon, \varepsilon'$ and $T$.

Combining the above estimates, we arrive at
\begin{eqnarray*}
&&\E \left[\sup_{t \leq s \leq T} e^{\beta s} | \overline{X}^{t,x,\varepsilon,\varepsilon'}(s)|^2+\sup_{t \leq s \leq T} e^{\beta s} | \overline{Y}^{t,x,\varepsilon,\varepsilon'}(s)|^2 + \int_t^T e^{\beta s} |\overline{Z}^{t,x,\varepsilon,\varepsilon'}(s)|^2 ds \right]\\
&\leq & C\left(\sqrt{\epsilon'}- \sqrt{\epsilon}\right)^2\nonumber\\
&&+ (33+ 144K)K e \max(1,T^2) \E \left[\sup_{t \leq s \leq T} e^{\beta s} | \overline{Y}^{t,x,\varepsilon,\varepsilon'}(s)|^2+\sup_{t \leq s \leq T} e^{\beta s} | \overline{Y}^{t,x,\varepsilon,\varepsilon'}(s)|^2\right.\\
&&\left. + \int_t^T e^{\beta s} |\overline{Z}^{t,x,\varepsilon,\varepsilon'}(s)|^2 ds \right] .  
\end{eqnarray*} 
Invoking assumption \eqref{cc}, the second term can be absorbed into the left-hand side. Consequently,
\begin{eqnarray*}
\E \left[\sup_{t \leq s \leq T} e^{\beta s} | \overline{Y}^{t,x,\varepsilon,\varepsilon'}(s)|^2+\sup_{t \leq s \leq T} e^{\beta t} | \overline{Y}^{t,x,\varepsilon,\varepsilon'}(s)|^2 + \int_t^T e^{\beta s} |\overline{Z}^{t,x,\varepsilon,\varepsilon'}(s)|^2 ds \right]  \leq C  (\sqrt{\varepsilon'} - \sqrt{\varepsilon})^2.
\end{eqnarray*}
\end{proof}

\section{Asymptotic behavior for delayed FBSDE}

\subsection{Convergence of distributions}
Let recall, for $\varepsilon>0$ and $(t,x) \in [0, T]\times \R$, $(X^{t,x,\varepsilon},Y^{t,x,\varepsilon},Z^{t,x,\varepsilon}$ the solution of FBSDE \eqref{Eq1} and define 
 \begin{eqnarray*}
 \left\{
 \begin{array}{l}
 u^{\varepsilon}(t, {\bf 1}(x))= Y^{t,x,\varepsilon}(t),\\\\
 v^{\varepsilon}(t, {\bf 1}(x))= Z^{t,x,\varepsilon}(t),
 \end{array}
 \right.
 \end{eqnarray*}
 where ${\bf 1}(x):[-T,0]\rightarrow\R$ such that ${\bf 1}(x)(u)=x$ for all $u\in [-T,0]$. Then $u^{\varepsilon}$ satisfies the following path-dependent PDE 
\begin{eqnarray}\label{PDE}
\left\{
\begin{array}{l}
\displaystyle \frac{\partial u^{\varepsilon} (t,{\bf 1}(x))}{\partial t}+b(t,{\bf 1}(x),{\bf 1}(u^{\varepsilon}(t,{\bf 1}(x))),{\bf 1}(v^{\varepsilon}(t,x)))\frac{\partial u^{\varepsilon} (t,{\bf 1}(x))}{\partial x}\\\\
+\frac{1}{2}a(t,{\bf 1}(x),{\bf 1}(u^{\varepsilon}(t,{\bf 1}(x))))\frac{\partial^2 u^{\varepsilon} (t,{\bf 1}(x))}{\partial x^2}\displaystyle+f\left(t,{\bf 1}(x),{\bf 1}(u^{\varepsilon}(t,{\bf 1}(x))),{\bf 1}(v^{\varepsilon}(t,{\bf 1}(x)))\right)=0\\\\
u^{\varepsilon}(T,{\bf 1}(x))=g({\bf 1}(x)),
 \end{array}
 \right.
 \end{eqnarray}
 
Since FBSDE \eqref{Eq1} has a unique solution, the following Markov property holds. For any $s \in [t, T]$, we have

\begin{eqnarray*}
\left\{
\begin{array}{l}
u^{\varepsilon}(s, X^{t, x,\varepsilon}_s)= Y^{s,X^{t, x, \varepsilon}_s}(s)= Y^{t,x,\varepsilon}(s)\\\\
 v^{\varepsilon}(s, X^{t, x, \varepsilon}_s)= Z^{s,X^{t, x, \varepsilon}_s}(s)= Z^{t,x,\varepsilon}(s),
 \end{array}
 \right.
 \end{eqnarray*}
 and it not difficult to to check that $(X^{t,x,\varepsilon},Y^{t,x,\varepsilon},Z^{t,x,\varepsilon})$ solves the following
decoupled FBSDEs
\begin{eqnarray}\label{Eq11}
\left\{
\begin{array}{lll}
\displaystyle	X^{t,x,\varepsilon}(s)&=& \displaystyle x+\int^{s\vee t}_{t} b(r,X_r^{t,x,\varepsilon},u^{\varepsilon}(r,X_r^{t,x,\varepsilon}),v^{\varepsilon}(r,X_r^{t,x,\varepsilon}))dr\displaystyle+\sqrt{\varepsilon}\int^{s\vee t}_t \sigma(r,X_r^{t,x,\varepsilon},u^{\varepsilon}(r,X_r^{t,x,\varepsilon}))dW(r) \\\\
\displaystyle	Y^{t,x,\varepsilon}(s)&=& \displaystyle g(X_T^{t,x,\varepsilon}) + \int_{s\vee t}^T f(r,X_r^{t,x,\varepsilon},u^{\varepsilon}(r,X_r^{t,x,\varepsilon}),v^{\varepsilon}(r,X_r^{t,x,\varepsilon}))dr - \int_{s\vee t}^T Z^{t,x,\varepsilon}(r) d W(r),\;\; , s\in [0,T].	
	\end{array}
	\right.
	\end{eqnarray}
\begin{rem}\label{R1}
Under assumptions $({\bf A3})$, one can prove like in Proposition 2.4 and Proposition B.6 of \cite{Delarue1} and Theorem 2.9 of \cite{Delarue2}, that there exist two constants $\kappa,\kappa_1$ only depending on $K, C, n, T$ (independent of $\varepsilon$) such that:
\begin{eqnarray}
|u^{\varepsilon}(t,{\bf 1}(x))|\leq \kappa,\label{Bound1}\\\nonumber\\
u^{\varepsilon}\in C^{1,2}_{b}([0,T]\times L_{-T}(\R))\nonumber\\\nonumber\\
\sup_{(t,x)\in [0,T]\times\R}\left|\frac{\partial u (t,{\bf 1}(x))}{\partial x}\right|\leq \kappa_1,\nonumber
\end{eqnarray}
and 
\begin{eqnarray*}
	Z^{t,x,\varepsilon}(s)=v^{\varepsilon}(s,X^{t,x,\varepsilon}_s)= \sqrt{\varepsilon}\nabla_{x}u^{\varepsilon}(s, X^{t,x,\varepsilon}_s)\sigma(s, X^{s,x,\varepsilon}_s, Y^{t,x,\varepsilon}_s).
\end{eqnarray*}
\end{rem}

From now on, we are concerned on the behavior laws of when $\varepsilon \rightarrow 0$.
\begin{theo}
Assume that conditions $({\bf A1})$-$({\bf A3})$ are satisfied.
\begin{description}
\item[(i)] For all $\delta > 0$
\begin{eqnarray*}
\lim_{\epsilon \rightarrow 0} \P\left(  \sup_{0\leq t \leq T} |X^\epsilon (t) - \mathcal{X} (t)| > \delta  \right)= 0. 	
\end{eqnarray*}
\item[(ii)] For $A\in \R$,  let define $\ Q^{\varepsilon}(A)=\P\left( Y^\varepsilon(.)^{-1}(A)\right)$. Then there exists a sub-sequence $Q^{\epsilon_n}$ which converges to $Q$  weakly in the Meyer-Zheng topology.
\end{description}
\end{theo}

\begin{proof}
Let set 
\begin{eqnarray}\label{N1}
\left\{
\begin{array}{l}
b^{\varepsilon}(s, {\bf 1}(x))= b(t,{\bf 1}(x),{\bf 1}(u^{\varepsilon}(t,{\bf 1}(x))),{\bf 1}(v^{\varepsilon}(t,x)))\\\\
\sigma^{\varepsilon}(s, {\bf 1}(x))= \sigma(t,{\bf 1}(x),{\bf 1}(u^{\varepsilon}(t,{\bf 1}(x))),{\bf 1}(v^{\varepsilon}(t,x)))\\\\
f^{\varepsilon}(s, {\bf 1}(x))= f(t,{\bf 1}(x),{\bf 1}(u^{\varepsilon}(t,{\bf 1}(x))),{\bf 1}(v^{\varepsilon}(t,x))).
 \end{array}
 \right.
 \end{eqnarray}
From the relation between $u^{\varepsilon}$ and $Y^{t,x,\varepsilon}$, between $v^{\varepsilon}$ and $Z^{t,x,\varepsilon}$, and between $u^{0}$ and $\mathcal{Y}^{t,x,\varepsilon}$, respectively, we have $\phi^{\varepsilon}$ converges uniformly to $\phi^{0}$ defined by
\begin{eqnarray*}
\sigma^0(s, {\bf 1}(x))= \sigma(s,{\bf 1}(x),{\bf 1}(u^{0}(s,{\bf 1}(x))))
\phi^{0}(s, {\bf 1}(x))= \phi(s,{\bf 1}(x),{\bf 1}(u^{0}(s,{\bf 1}(x))),0), 
 \end{eqnarray*}
with $\phi=b,\, f$.

It follows from the definition of $X^{t,x,\varepsilon}$ and $\mathcal{X}^{t,x}$ with respect to equation \eqref{Eq11}, that
\begin{eqnarray*}
|X^{t,x,\varepsilon}(s) - \mathcal{X}^{t,x}(s)|&\leq & \left| \int_t^s (b^{\varepsilon}(r,X^{t,x,\varepsilon}_r)- b^{0}(r,\mathcal{X}^{t,x}_r))dr \right| + \sqrt{\varepsilon} \sup_{ t\leq s \leq T} \left|\int_t^s\sigma(r, X_r^{t,x,\varepsilon})dW(r)\right| \\
 &\leq & \sup_{ t\leq s \leq T} \left| \int_t^s b^{\varepsilon}(r,X^{t,x,\varepsilon}_r)- b^{0}(r,\mathcal{X}^{t,x}_r))dr\right| + \sqrt{\varepsilon} \sup_{ t\leq s \leq T}\left|\int_t^s\sigma(r, X_r^{t,x,\varepsilon})dW(r)\right|.
\end{eqnarray*}
Thus, by passing to the probability, we obtain
\begin{eqnarray}\label{Q1}
\P(\sup_{ t\leq s \leq T} | X^{t,x,\varepsilon}(s) - \mathcal{X}^{t,x}(s)|> \delta )& \leq &\P\left(\left| \int_t^T (b^{\varepsilon}(r,X^{t,x,\varepsilon}_r)- b^{0}(r,\mathcal{X}^{t,x}_r))dr\right| > \frac{\delta}{2} \right) \nonumber\\
 &&+ \P\left(  \sup_{ t\leq s \leq T}\left|\sqrt{\varepsilon}\int_t^s\sigma(r, X_r^{t,x,\varepsilon})dW(r)\right|>\frac{\delta}{2} \right).
\end{eqnarray}
From Chebyshev's inequality and the first assertion of lemma \eqref{l1} we have
\begin{eqnarray}\label{Q2}
\P\left( \left| \int_t^T (b^{\varepsilon}(r,X^{t,x,\varepsilon}_r)- b^{0}(r,\mathcal{X}^{t,x}_r))dr \right| > \frac{\delta}{2} \right) &\leq & \frac{4}{\delta^2} \E\left[\left| \int_t^T (b^{\varepsilon}(r,X^{t,x,\varepsilon}_r)- b^{0}(r,\mathcal{X}^{t,x}_r))dr \right|^2 \right]\nonumber\\
&\leq & \frac{4T}{\delta^2} \E\int_t^T |b^{\varepsilon}(r,X^{t,x,\varepsilon}_r)- b^{0}(r,\mathcal{X}^{t,x}_r)|^2ds \nonumber\\
&\leq & C\sqrt{\varepsilon}.
	\end{eqnarray}
For estimation of the second term in \eqref{Q1}, we use Chebyshev's inequality both Burkholder-Davis-Gundy and Lemma \eqref{l1} we have
\begin{eqnarray}\label{Q3}
\P\left(  \sup_{ t\leq s \leq T}\left|\sqrt{\varepsilon}\int_t^s\sigma(r, X_r^{t,x,\varepsilon})dW(r)\right|>\frac{\delta}{2} \right)&\leq & 4\varepsilon\E\left(\int_t^T|\sigma(r, X_r^{t,x,\varepsilon})|^2dr\right)\nonumber\\
&\leq & C\varepsilon.
\end{eqnarray}
Finally putting \eqref{Q1} and \eqref{Q2} in  \eqref{Q1} and letting $\varepsilon$ to $+\infty$ we derive $(i)$.
	
Concerning $(ii)$, the argument is similar to that in \cite{Cal}. Thus, we first
show that the solution of the FBSDE \eqref{Eq11} is a quasimartingale.
Indeed, given a subdivision $\pi:0=t_0<\cdots<t_n=T$, and setting
\begin{eqnarray*}
 V^{\pi}_{T}(Y^{t,x,\varepsilon})=\sum_{i=0}^{n-1}\E(|	Y^{t,x,\varepsilon}(t_{i+1})-	Y^{t,x,\varepsilon}(t_{i})|\mathcal{F}_{t_i}),
\end{eqnarray*}
one can derive that
\begin{eqnarray*}
V^{\pi}_{T}(Y^{t,x,\varepsilon})\leq C.
\end{eqnarray*}
 Noting that $V_T(Y^{t,x,\varepsilon}) = \sup_{\pi}V^{\pi}_{T}(Y^{t,x,\varepsilon}) < +\infty$, the result follows.	
Since $D(\R)= D([0,T]; R)$, the Skorohod space of right continuous with
left limits functions $\varphi$ on $[0, T]$ with values in $\R$ such that
\begin{eqnarray*}
\psi(T-)=\lim_{t\rightarrow T}\psi(t)=\psi(T),
\end{eqnarray*}
with by convention $\psi(0-) = 0$, is a Lusin space, there exist a compact metric space $K$ such that $D(R^n)$ is a Borel set in $K$. Next, the set of probability $\widetilde{\Q}^{\varepsilon}$ on  $K$ defined by 
\begin{eqnarray*}
\widetilde{\Q}^{\varepsilon}(A)=\Q^{\varepsilon}(A\cap D(\R^n)).
\end{eqnarray*}
for all $A\in \mathcal{B}(K)$, is compact for the weak convergence. Hence there exist a subsequence of $\widetilde{\Q}^{\varepsilon}$ denoted $\widetilde{\Q}^{\varepsilon_n}$ and a probability measure $\widetilde{\Q}$ on $K$ such that
\begin{eqnarray*}
	\widetilde{\Q}^{\varepsilon_n}  \stackrel{w}{\rightarrow} \widetilde{\Q}.
\end{eqnarray*}
On the other hand, in view of it definition, $\widetilde{\Q}^{\varepsilon}$ is the distribution of $Y^{t,x,\varepsilon}$ considered as random variable on $(K,\mathcal{B}(K))$. Furthermore, by Proposition 3.1 in \cite{Cal}, one has possibly along a subsequence, $\widetilde{\Q}^{\varepsilon_n}$  converges weakly to a probability law $\Q^{*}\in \mathcal{M}(D(\R))$. The uniqueness of the weak limit implies that  
\begin{eqnarray*}
	\Q^{*}(A)=\widetilde{\Q}(A), \;\;\; \forall\;\; A\in \mathcal{B}(D(\R)).
\end{eqnarray*}
In particular
\begin{eqnarray*}
	1=\Q^{*}(D(\R))=\widetilde{\Q}(D(\R)),
\end{eqnarray*}
which end the proof.
\end{proof}

\subsection{Convergence almost surely}

\begin{theo}
Assume that conditions $({\bf A1})$-$({\bf A3})$ are satisfied. We have 
\begin{enumerate}
\item For each $s, t \in  [0, T],\; t \leq s$, $(X^{t,x,\varepsilon},Y^{t,x,\varepsilon},Z^{t,x,\varepsilon})$ the solution of \eqref{Eq1} converges in $\mathcal{S}^2(\R)\times\mathcal{S}^2(\R)\times\mathcal{H}^2(\R^n)$ when $\varepsilon\rightarrow 0$, to $(\mathcal{X}^{t,x},\mathcal{Y}^{t,x},0)$ solution of \eqref{u}.
\item Let denote $Y^{t,x}(t)$ the limit of $Y^{\varepsilon,t,x}(t)$ when $\varepsilon \rightarrow 0$ and set $u(t,{\bf 1}(x))=Y^{t,x}(t)$. Then the function $u$ is a viscosity solution of PDE the following path-dependent PDE 
\begin{eqnarray}\label{u1}
\left\{
\begin{array}{l}
\displaystyle \frac{\partial u (t,{\bf 1}(x))}{\partial t}+b(t,{\bf 1}(x),{\bf 1}(u(t,{\bf 1}(x))),{\bf 1}(0))\frac{\partial u(t,{\bf 1}(x))}{\partial x}\\\\\displaystyle+f\left(t,{\bf 1}(x),{\bf 1}(u(t,{\bf 1}(x))),{\bf 1}(0)\right)=0\\\\
u(T,{\bf 1}(x))=g({\bf 1}(x)),
 \end{array}
 \right.
 \end{eqnarray}
\item The function $u$ is bounded, continuous Lipschitz in $\varphi$ and uniformly continuous in time.
\item Furthermore, if $u\in C^{1,1}_{b}([0, T]\times  L_{-T}^{\infty}(\mathbb{R}))$, since \eqref{u} has a unique continuous solution, the function $u$ is a classical solution of \eqref{u1}.
\end{enumerate}
\end{theo}

\begin{proof}
\begin{itemize}
\item [(i)] Recalling Lemma \ref{l4}, one derive that for a fixed $(t,x)\in [0,T]\times\R,\;(X^{t,x,\varepsilon},Y^{t,x,\varepsilon})$, solution of \eqref{Eq1}, is a Cauchy sequence on the Banach space $\mathcal{S}^2(\R)\times\mathcal{S}^2(\R)\times$. Therefore it converges to the  processus $(\mathcal{X}^{t,x}, \mathcal{Y}^{t,x})$. On the other hand, since $\sigma$ is bounded, it follows from Remark \ref{R1} that $Z^{t,x,\varepsilon}(s)$ converge to $0$. Therefore $(X^{t,x,\varepsilon},Y^{t,x,\varepsilon},Z^{t,x,\varepsilon})$ converges to $(\mathcal{X}^{t,x}, \mathcal{Y}^{t,x}, 0)$.

Next, under assumption $({\bf A4})$, passing to the point-wise limit as $\varepsilon\rightarrow 0$ in the forward equation of\eqref{Eq1} yields
\begin{eqnarray*}
\mathcal{X}^{t,x}(s)=x+\int^{s}_{t}b(r,\mathcal{X}^{t,x}(r),\mathcal{Y}^{t,x}(r),0)dr,\;\; s\in[t,T].
\end{eqnarray*}
Moreover, by Itô's isometry,
\begin{eqnarray*}
\E\left(\left|\int^T_sZ^{t,x,\varepsilon}(r)dW(r)\right|^2\right)=\E\left(\int^T_s|Z^{t,x,\varepsilon}(r)|^2dr\right)\rightarrow 0,\;\;\;\; \mbox{as}\;\; \varepsilon\rightarrow 0.
\end{eqnarray*}
Hence, 
\begin{eqnarray*}
\int^T_sZ^{t,x,\varepsilon}(r)dW(r)\rightarrow 0,\;\P-\mbox{a.s};\;\;\;\; \mbox{as}\;\; \varepsilon\rightarrow 0.
\end{eqnarray*}
Combining this convergence with the continuity of $f$ and $g$, we deduce that the backward equation in \eqref{Eq1} converges to 
\begin{eqnarray*}
\mathcal{Y}^{t,x}(s)=g(\mathcal{X}^{t,x}_T)+\int_{s}{T}f(r,\mathcal{X}^{t,x}(r),\mathcal{Y}^{t,x}(r),0)dr, \;\;\; [t,T].
\end{eqnarray*}
$\varepsilon\rightarrow 0$. Therefore, the pair $(\mathcal{X}^{t,x},\mathcal{Y}^{t,x})$ is a solution of \eqref{u}.
\item [(ii)] Combining Lemma \ref{l1} and Lemma \ref{l3} together with the definition of $u^{\varepsilon}$, we obtain
\begin{eqnarray}\label{Uniforme}
|u^{\varepsilon}(t,{\bf 1}(x))-u^{\varepsilon}(t',{\bf 1}(x'))|^{2}\leq \alpha|x-x'|^2+\beta(1+|x|^2)|t-t'|^2,
\end{eqnarray}
where $\alpha,\; \beta >0$ are constants depends only $K$ and $\Lambda$.
This estimate shows that the family $u^{\varepsilon}$ is equicontinuous on every compact subset of  $[0,T]\times L_{-T}(\R^n)$. Therefore, by the Arzela-Ascoli theorem, we conclude that $u^{\varepsilon}$ converges uniformly on $[0, T]\times \mathcal{K}$, for every compact subset $\mathcal{K}\subset L_{-T}(\R^n)$, toward the function $u$ defined by 
\begin{eqnarray*}
u(t, {\bf 1}(x)) = Y^{t,x}(t).
\end{eqnarray*}
Passing to the limit as $\varepsilon 0$ in \eqref{Uniforme}, we obtain
\begin{eqnarray}\label{Uniforme}
|u(t,{\bf 1}(x))-u(t',{\bf 1}(x'))|^{2}\leq \alpha|x-x'|^2+\beta(1+|x|^2)|t-t'|^2,
\end{eqnarray}
for all $(t, x), (t', x')$ in $[0, T]\times \R^n$.  Consequently, the limit function u is Lipschitz continuous with respect to the spatial variable x and uniformly continuous with respect to the time variable t. Moreover, the boundedness of $u$ follows directly from \eqref{Bound1}, since the latter estimate is uniform with respect to $\varepsilon$. Furthermore, by Theorem 9 in \cite{card1} or Theorem 11 \cite{card2}, the function $u^{\varepsilon}$ is a viscosity solution on $[0, T]\times L_{-T}(\R^n)$ of \eqref{u}. Since the coefficients of the associated quasilinear parabolic system are Lipschitz continuous, the stability property of viscosity solutions under locally uniform convergence applies (see \cite{Cral}). Therefore, using the compact uniform convergence of $u^{\varepsilon}$ toward $u$, we conclude that $u$ is a viscosity solution of \eqref{u1}.

Moreover, let $v: [0, T]\times L_{-T}(\R)\rightarrow \R$ be a solution of class $C^{1,1}_b([0, T]\times L_{-T}(\R), \R)$ to \eqref{u1}, which is Lipschitz continuous with respect to the spatial variable and uniformly continuous with respect to the time variable. Fix (t,x) $(t,x)\in [0, T]\times\R$, and define
\begin{eqnarray*}
	h:[t,T]\rightarrow \R , \;\; h(s)=v(s, \mathcal{X}^{t,x}_s).
\end{eqnarray*}
Using the chain rule in the path space setting, we obtain
\begin{eqnarray*}
\frac{dh(s)}{ds}=\frac{\partial v(s,\mathcal{X}^{t,x}_s)}{\partial s}+\mathcal{S}(v)(s,\mathcal{X}^{t,x}_s)+b(s,\mathcal{X}^{t,x}_s,\mathcal{Y}^{t,x}_s,\mathcal{Z}^{t,x}_s)\frac{\partial v(s,\mathcal{X}^{t,x}_s)}{\partial x},
\end{eqnarray*}
where $\mathcal{S}$ est l'opérateur de translation associé au segment $X_s$. 
Evaluating the above identity at $s=t$, and observing that, we have $\mathcal{S}(v)(t,\mathcal{X}^{t,x}_t)=0$, since $\mathcal{X}^{t,x}_t$ is constant on $[-T,0]$, we obtain
\begin{eqnarray}\label{Z1}
\frac{dh(s)}{ds}&=&\frac{\partial v(t,\mathcal{X}^{t,x}_t)}{\partial t}+b(t,\mathcal{X}^{t,x}_t,\mathcal{Y}^{t,x}_t,Z^{t,x}_t)\frac{\partial v(t,\mathcal{X}^{t,x}_t)}{\partial x}\nonumber\\
&=&-f(t,\mathcal{X}^{t,x}_t,\mathcal{Y}^{t,x}_t,Z^{t,x}_t).
\end{eqnarray}
Next, recalling  $\mathcal{X}^{t,x}_t={\bf 1}(x),\; \mathcal{Y}^{t,x}_t={\bf 1}(\mathcal{Y}^{t,x}(t))$ and $\mathcal{Z}^{t,x}_t={\bf 1}(0)$, equation \eqref{Z1} reduces to
\begin{eqnarray*}
\frac{dh(s)}{ds}&=&\frac{\partial v(t,{\bf 1}(x))}{\partial t}+b(t,{\bf 1}(x),{\bf 1}(\mathcal{Y}^{t,x}(t)),{\bf 1}(0))\frac{\partial v(s,\mathcal{X}^{t,x}_s)}{\partial x}\nonumber\\
&=&-f(t,{\bf 1}(x),{\bf 1}(\mathcal{Y}^{t,x}(t)),{\bf 1}(0)).
\end{eqnarray*}
Moreover, 
\begin{eqnarray*}
	v(T,X^{T,x}_T)=g(X^{T,x}_T)=g({\bf 1}(x)).
\end{eqnarray*}
Therefore, under the assumption of uniqueness for the system of ordinary differential equations associated with \eqref{u1}, we conclude that
\begin{eqnarray*}
	v(t,{\bf 1}(x))=u(t,{\bf 1}(x)).
\end{eqnarray*}
Consequently, equation \eqref{u1} admits at most one solution in the class $C^{1,1}_b([0, T]\times \R^n)$ of functions which are Lipschitz continuous with respect to the spatial variable and uniformly continuous with respect to the time variable. 
\end{itemize}
\end{proof}

\subsection{Large deviation principle}

In this section, we study the Freidlin-Wentzell's large deviation principle for the laws of the family of processes $(X^{t,x,\varepsilon}(.), Y^{t,x,\varepsilon}(.))_{\varepsilon\in (0,1]}\in \mathcal{S}^{2}(t,T,\R)\times \mathcal{S}^{2}(t,T,\R)$, where $(X^{t,x,\varepsilon}(.), Y^{t,x,\varepsilon}(.), Z^{t,x,\varepsilon}(.))$ is the unique solution of equation \eqref{Eq1}.

In order to avoid an unnecessarily lengthy exposition, we shall rely on the definitions associated with this notion as given in Definitions 4.1 and 4.2 of \cite{Cal}. Next we recall the following result from Theorem 3. 1 in  \cite{MA}.
\begin{propo}
Assume that conditions $({\bf A1})$-$({\bf A3})$ are satisfied and the notation of \eqref{N1}. Then, the family $(X^{\varepsilon,t,x}(.) : 0 <\varepsilon< 1)$ of random variables of the solutions of the following perturbed delayed SDE 
\begin{eqnarray*}
X^{t,x,\varepsilon}(s)&=& x+\int^{s\vee t}_{t} b^{\varepsilon}(r,X_r^{t,x,\varepsilon})dr+\sqrt{\varepsilon}\int^{s\vee t}_t \sigma^{\varepsilon}(r,X_r^{t,x,\varepsilon})dW(r) 
\end{eqnarray*}
obey a large deviations principle in $C([t, T]; R^n)$, with the good rate function $I$ defined as
\begin{eqnarray*}
I_1(\phi)=
\left\{
\begin{array}{ll}
\displaystyle\inf\left\{\frac{1}{2}\int^T_t\|\varphi(s)\|^2ds,\;\; \phi'(s)=b(s,\phi(s))+\sigma(s,\phi(s))\varphi(s),\; s\in [t,T]\right\}&,\;\; \varphi\in L^2([t,T],\R^n)\\\\
+\infty	& otherwise.
\end{array}
\right.
\end{eqnarray*}
Moreover, the level sets of $I$ are compact and for every Borel subset $A$ of $C([t, T]; \R^n)$, we have
\begin{eqnarray*}
-\inf_{g\in }I(g)\leq \liminf_{\varepsilon\rightarrow 0}\varepsilon\log(\P (X^{t,x,\varepsilon}\in A))\leq-\inf_{g\in \overline{A} }I(g)
\end{eqnarray*}
\end{propo}
For the proof we refer the reader to \cite{MA}.

We conclude by establishing a large deviation principle for the backward component of system \eqref{Eq11}. To this end, we consider the function $$F^{\varepsilon}, F:C([t, T]; \R^n)\rightarrow C([t, T];\R^n)$$ defined respectively by  
\begin{eqnarray*}
F^{\varepsilon}(\psi)(s)=u^{\varepsilon}(s,\psi_s)
\end{eqnarray*}
and 
\begin{eqnarray*}
F(\psi)(s)=u(s,\psi_s)
\end{eqnarray*}
where $\psi_s=(\psi(s+\theta))_{\theta\in [-T,0]}$ belongs in $L_{-T}(\R^n)$,$u^{\varepsilon}$ and $u$ are given by \eqref{PDE} and \eqref{u1} respectively. As we have seen before we have 
 \begin{eqnarray*}
 Y^{\varepsilon,t,x}(t)=F^{\varepsilon}(X^{\varepsilon,t,x})(t)\;\;\mbox{and}\;\;	\mathcal{Y}^{t,x}(t)=F(\mathcal{X}^{t,x})(t)
 \end{eqnarray*}
 Next, in order to study the convergence of $F^{\varepsilon}$ to $F$, we set the following
\begin{eqnarray*}
	\|F^{\varepsilon}(\psi)-F(\psi)\|=\sup_{s\in [t,T]}|u^{\varepsilon}(s,\psi_s)-u(s,\psi_s)|,\;\; \psi\in C([t, T];\R^n).
\end{eqnarray*}
or 
\begin{eqnarray*}
\|F^{\varepsilon}(\psi)-F(\psi)\|=\sup_{s\in [t,T]}|Y^{\varepsilon,s,\psi(s)}(s)-\mathcal{Y}^{s,\psi(s)}(s)|,\;\; \psi\in C([t, T];\R^n).
\end{eqnarray*}
We have 
\begin{theo}
Assume that conditions $({\bf A1})$-$({\bf A3})$ are satisfied. Then, the family $(Y^{\varepsilon,t,x}(.))_{0<\varepsilon<1})$ satisfy, as $\varepsilon$ goes to $0$, a large deviation principle with a rate function
\begin{eqnarray*}
	I_2(\psi) = \inf\left\{I_1(\phi): F(\phi)(s) = \psi(s) = u(s, \phi(s)),\;\; s\in  [t, T]\right\}.
\end{eqnarray*}
\end{theo}
\begin{proof}
By virtue of the contraction principle (see Lemma 4.1 in \cite{Cal} or Lemma 3.1 in \cite{MA} for general case), we just need to show that for all $\varepsilon\in (0, 1]$, the function $F^{\varepsilon}$ is continuous and the family $(F{\varepsilon})_{\varepsilon\in(0,1]}$  converges uniformly to $F$ on every compact of $C([0, T]; \R^n)$, as $\varepsilon$ tends to zero.As the proof proceeds along essentially the same lines as that of Theorem 3.3 in \cite{MA}, we refer the reader to that result for the relevant details. Consequently, the proof is omitted here.
\end{proof}

\begin{rem}
 Our approach is inspired by the PDE method developed by Cruzeiro et al.~\cite{Cal}, which relies on the decoupling of the fully coupled FBSDE through the associated semilinear parabolic equation. In the delayed setting, however, the presence of memory destroys the finite-dimensional Markov property of the forward component. To recover a Markovian formulation, one has to enlarge the state space by introducing the segment process $ X_t=(X(t+\theta))_{\theta\in[-T,0]}$, which takes values in the infinite-dimensional space $C([-T,0];\mathbb{R})$. Consequently, the associated decoupling field is naturally characterized as the solution of a semilinear PDE posed on \[ [0,T]\times C([-T,0];\mathbb{R}^n). \] Under assumptions $(\mathbf{A1})$-(ii) and $(\mathbf{A2})$-(ii), the delayed coefficients depend only on constant trajectory segments. Defining \[ \bar u^\varepsilon(t,x)=u^\varepsilon(t,\mathbf{1}(x)), \] together with \[ \bar{\phi}(t,x,\bar u^\varepsilon,\bar v^\varepsilon) =\phi\!\left(t,\mathbf{1}(x),\mathbf{1}(\bar u^\varepsilon), \mathbf{1}(\bar v^\varepsilon)\right),\qquad \phi\in\{b,f\}, \] \[ \bar{\sigma}(t,x,\bar u^\varepsilon) =\sigma\!\left(t,\mathbf{1}(x),\mathbf{1}(\bar u^\varepsilon)\right), \qquad \bar g(x)=g(\mathbf{1}(x)), \] the infinite-dimensional PDE \eqref{PDE} reduces to the classical finite-dimensional semilinear equation \[ \left\{ \begin{aligned} &\frac{\partial\bar u^\varepsilon}{\partial t} +\bar b\,\frac{\partial\bar u^\varepsilon}{\partial x} +\frac12\bar a\,\frac{\partial^2\bar u^\varepsilon}{\partial x^2} +\bar f=0,\\ &\bar u^\varepsilon(T,x)=\bar g(x). \end{aligned} \right. \] The analysis of the fully infinite-dimensional problem, where the coefficients genuinely depend on the entire past trajectory, requires substantially different techniques and will be addressed in future work. \end{rem}


\bigskip
	
\begin{thebibliography}{99}   

\bibitem{AHJ} A. Aman, H. Coulibaly and J. Djordjevic, General fully coupled forward-backward stochastic differential equations with delayed generator, \emph{Stoch. Dyn.} \textbf{23} (2023), no.\textasciitilde 2, 2350012.

\bibitem{Antonelli} F. Antonelli, Backward-forward stochastic differential equations, \emph{Ann. Appl. Probab.} \textbf{3} (1993), pp. 777-793. 

\bibitem{card1} F. Cordoni, L. Di Persio, A BSDE with delayed generator approach to pricing under counterparty risk and collateralization, \emph{International Journal of Stochastic Analysis} (2016), Article ID 1059303, 11 pages.

\bibitem{card2} F. Cordoni, L. Di Persio, L. Maticiuc, and A. Zalinescu, A stochastic approach to path-dependent nonlinear Kolmogorov equations via BSDEs with time-delayed generators and applications to finance, \emph{Stochastic Process. Appl.} \textbf{130} (2020),  1669-1712.


\bibitem{Cral} M.G. Crandall, H. Ishii and P.-L. Lions, User's guide to viscosity solutions of second order partial differential equations,\emph{Bull. Amer. Math. Soc. (N.S.)} \textbf{27} (1992), 1--67.
		
\bibitem{Cal} A.B. Cruzeiro, A.O. Gomes and L. Zhang, Asymptotic properties of coupled forward-backward stochastic differential equations, \emph{Stoch. Dyn.} \textbf{14} (2014), no.\textasciitilde 3, 1450004.
 
\bibitem{Delarue1} F. Delarue, On the existence and uniqueness of solutions to FBSDEs in a nondegenerate case, \emph{Stochastic Process. Appl.} \textbf{99} (2002), 209--286.

\bibitem{Delarue2} F. Delarue, Estimates of the solutions of a system of quasi-linear PDEs, in A Probabilistic Scheme, Séminaire des Probabilités XXXVII, Lecture Notes in Mathematics, Vol. 1832 (Springer, 2003), pp. 290-332.
		
\bibitem{D1} L. Delong, BSDEs with time-delayed generators of a moving average type with applications to non-monotone preferences, \emph{Stochastic Models} \textbf{28} (2012), 281--315.

\bibitem{D2} L. Delong, Applications of time-delayed backward stochastic differential equations to pricing, hedging and portfolio management \emph{Appl. Math.} \textbf{39} (2012), 463--488.
		
\bibitem{DI1} L. Delong and P. Imkeller, Backward stochastic differential equations with time-delayed generators: results and counterexamples,
\emph{Ann. Appl. Probab.} textbf{20} (2010), no.\textasciitilde 4, 1512--1536.
		
\bibitem{DI2}L. Delong and P. Imkeller, On Malliavin differentiability of BSDEs with time-delayed generators driven by Brownian motions and Poisson random measures, \emph{Stochastic Process. Appl.} \textbf{120} (2010), 1748--1775.
		
		
\bibitem{el1} N. El Karoui, E. Pardoux and M.C. Quenez, Reflected backward SDEs and American options, in: D. Dempster and S. Pliska (Eds.), \emph{Numerical Methods in Finance}, Publications of the Newton Institute, Cambridge University Press, Cambridge, 1997, pp. 215--231.
		
\bibitem{Kal} N. El Karoui, S. Peng and M.C. Quenez, Backward stochastic differential equations in finance, \emph{Math. Finance} \textbf{7} (1997), no.\textasciitilde 1, 1--71.
		
\bibitem{KJL} N. El Karoui, M. Jeanblanc and V. Lacoste, Optimal portfolio management with American capital guarantee, \emph{J. Econom. Dynam. Control} \textbf{29} (2005), no.\textasciitilde 3, 449--468.


		
\bibitem{el3} N. El Karoui and S. Hamadène, BSDEs and risk-sensitive control, zero-sum and nonzero-sum game problems of stochastic functional differential equations, \emph{Stochastic Process. Appl.} \textbf{107} (2003), no.\textasciitilde 1, 145--169.


\bibitem{HP}  Y. Hu and S. Peng, Solution of forward-backward stochastic differential equations, \emph{Probab. Theory Related Fields} \textbf{103} (1995), 273--283.
		
\bibitem{MA} C. Manga and A. Aman, Large deviations for stochastic differential equations with the general delayed generator, \emph{Random Oper. Stoch. Equ.} \textbf{28} (2020), no.\textasciitilde 3, 197--207.

\bibitem{MAT}
C. Manga, A. Aman and N. Tuo, Asymptotic behavior for delayed backward stochastic differential equations, \emph{Commun. Statist. Simul. Comput.} \textbf{54} (2025), no.\textasciitilde 7, 2491--2506.

\bibitem{MPY} J. Ma, P. Protter and J. Yong, Solving forward-backward stochastic differential equations explicitly: A four-step scheme, \emph{Probab. Theory Related Fields} \textbf{98} (1994), 339-359.

\bibitem{MY} J. Ma, J. Yong, Solvability of forward-backward SDE's and the nodal set of Hamilton-Jacobi-Bellman equations, \emph{Chin. Ann. Math.} \textbf{16B} (1995), pp. 279-298 
		

\bibitem{Ma} J. Ma, Z. Wu, D. Zhang and J. Zhang, On well-posedness of forward-backward SDEs: A unified approach, \emph{Ann. Appl. Probab.} \textbf{25} (2015), no.\textasciitilde 4, 2168--2214.

\bibitem{PW} S. Peng and Z. Wu, Fully coupled forward-backward stochastic differential equations and applications to optimal control, \emph{SIAM Journal on Control and Optimization} \textbf{37} (1999), 825-843.
		
\end{thebibliography}
\end{document}